\documentclass{amsart}
\usepackage{amsmath, amsthm, amssymb}
\usepackage{amssymb}
\usepackage{epsfig}
\usepackage{amsfonts}
\usepackage{amsmath}
\usepackage{epsfig}
\usepackage{tikz}
\usetikzlibrary{matrix}

\theoremstyle{plain}
\newtheorem{theorem}{Theorem}[section]
\newtheorem{proposition}[theorem]{Proposition}
\newtheorem{lemma}[theorem]{Lemma}

\newtheorem{corollary}[theorem]{Corollary}

\theoremstyle{definition}
\newtheorem{definition}[theorem]{Definition}
\newtheorem{example}[theorem]{Example}
\theoremstyle{remark}
\newtheorem{remark}[theorem]{Remark}

\def\C{\mathbb C}

\def\P{\mathbb P}

\def\O{\mathcal O}

\def\F{\mathcal F}
\def\Sym{{\rm Sym}}
\def\Sec{{\rm Sec}}

\def\Conf{{\rm Conf}}

\begin{document}

\title[Fourier-Deligne transform and representations of $S_n$]{Fourier-Deligne transform and representations
of the symmetric group}

\author{Galyna Dobrovolska}

\begin{abstract}
We calculate the Fourier-Deligne transform of the IC extension to ${\C}^{n+1}$ of the local system ${\mathcal L}_{\Lambda}$
on the cone over $\Conf_n({\P}^1)$ associated to a representation $\Lambda$ of $S_n$, where the length $n-k$ of the first row of 
the Young diagram of $\Lambda$ is at least $\frac{|\Lambda|-1}{2}$. The answer is the IC extension to the dual vector space
 ${\C}^{n+1}$ 
of the local system ${\mathcal R}_{\lambda}$ on the cone over the $k$-th secant variety of the rational normal curve in ${\P}^n$, 
where ${\mathcal R}_{\lambda}$ corresponds to the representation $\lambda$ of $S_k$, the Young diagram of which is obtained from
the Young diagram of $\Lambda$ by deleting its first row. We also prove an analogous statement for $S_n$-local systems on fibers of the
Abel-Jacobi map. We use our result on the Fourier-Deligne transform to rederive a part of a result of Michel Brion on Kronecker coefficients.
\end{abstract}

\maketitle

\section{Introduction}

Let $W \cong \C^2$ be a two-dimensional complex vector space. For any positive integer $r$ we can identify the projective 
space ${\P}^r=\P({\Sym}^r(W))$ with degree $r$ polynomials in two variables up
to scaling, and consider the discriminant hypersurface $D_r$ in ${\P}^r$ which is formed by polynomials with repeated roots.
The complement ${\P}^r - D_r$ is identified with the configuration space $\Conf_r(\P^1)$ of $r$ unordered distinct
points on the projective line ${\P}^1$. Hence ${\pi}_1({\P}^r - D_r)$ is the spherical braid group $B_r({\P}^1)$. It is known
that $B_r({\P}^1)=B_r/ (b_1 b_2 ... b_{r-1} b_{r-1} ... b_2 b_1) $, where $b_1, b_2, ..., b_{r-1}$ are the standard
generators of the braid group $B_r$, and the usual map $B_r \to S_r$ factors through $B_r({\P}^1)$.
Hence a representation $\rho$ of the symmetric group $S_r$ gives a representation of $B_r({\P}^1)$
 and thus a local system $\mathcal{L}_{\rho}$ on ${\P}^r - D$. We denote by ${\bf L}_{\rho}$ the perverse sheaf on the complex 
vector space ${\Sym}^r(W)$
which is the intermediate extension to ${\Sym}^r(W)$ of the lift of ${\mathcal L}_{\rho}$ to the complement of the cone over $D_r$ in ${\Sym}^r(W)$. 
The goal of this paper is to find the Fourier-Deligne transform (see Definition \ref{fourierdefn} below or Chapter III of \cite{KS}) of such objects ${\bf L}_{\rho}$.

The Fourier-Deligne transform for similar local systems on the set of semisimple regular matrices 
is important in Springer theory, see \cite{B}, Section XI,  or \cite{KW}, Chapter VI, and for similar local systems on 
curves of higher genus it is important in the geometric Langlands program in the paper \cite{D} of Drinfeld (1983).

The main result of this paper is Theorem \ref{main} in which we calculate the Fourier-Deligne transform of ${\bf L}_{\rho}$ where 
the first row of $\rho$ has length at least $\frac{|\rho|-1}{2}$ (in this case we say that $\rho$ has a long first row); 
Theorem \ref{main} is generalized to the case of an arbitrary smooth curve embedded into projective space in Theorem 
\ref{abeljacobi}. Theorem \ref{main} is also used to derive Proposition \ref{brion} in which we find
a  relation with Michel Brion's result in \cite{Brion}, Cor. 2, Section 3.4.

\subsection*{Main theorem for $\P^1$}
 For a Young diagram denoted by a
capital Greek letter, we denote the Young diagram obtained by deleting its first row by the corresponding
lowercase Greek letter. (Note that we use the same notation for a representation of $S_r$ and for its Young diagram).
Let $|\Lambda|=n$ and let $k=|\lambda|$. Consider the rational normal curve
$X$  in the projective space $\P(\Sym^n (W^*))$ which is the dual variety of the hypersurface $D_n \subset \P(\Sym^n (W))$. 
The $k$-th secant variety $\Sec^k=\Sec^k (X)$ of $X$
 is the closure of the union of "$k$-secant" $(k-1)$-planes $\P^{k-1}$ that pass through $k$ distinct points of $X$.
When $k\leq \frac{n+1}{2}$ (i.e. $\Lambda$ has a long first row) 
the variety $\Sec^k - \Sec^{k-1}$ is smooth and has the structure of a fiber bundle 
over the unordered configuration space ${\P}^k - D_k \cong {\rm Conf}_k(\P^1)$ (see Lemma \ref{secantlemma});
the cone $U_k$ (with vertex extracted) over $\Sec^k - \Sec^{k-1}$ in $\Sym^n(W^*)$ shares the same properties. We define a local system ${\mathcal R}_{\lambda}$ on $U_k$  by pulling back the local system ${\mathcal L}_{\lambda}$ from
${\P}^k-D_k$ via the bundle map. 
We denote the perverse sheaf which is the intermediate extension of ${\mathcal R}_{\lambda}$
 to $\Sym^n (W^*)$ by ${\bf R}_{\lambda}$.

\begin{theorem} \label{main} 
If $\Lambda$ has a long first row,
the Fourier-Deligne transform of the perverse sheaf ${\bf L}_{\Lambda}$ 
on the vector space $\Sym^n(W)$ 
is the perverse sheaf ${\bf R}_{\lambda}$ on $\Sym^n(W^*)$. 
\end{theorem}

\begin{example} \label{regrep}
Let $\Lambda$ be the standard representation of $S_n$. In this case we can verify Theorem \ref{main}
using the facts that 1) the Fourier transform $\Phi$ agrees with the Radon transform $R$ up to constant complexes, and 2) 
 the Radon transform of a complex $\mathcal F$ on a subvariety $Y$ in $\P^n$ is given by monodromy of hyperplane sections of the cone over $Y$ with coefficients in the lift of $\mathcal F$ (see \cite{B}). 

Take $Y=X$, the rational normal curve in $\P(V^*)$, and the constant sheaf on $X$. A generic 
hyperplane section of the cone over $X$ in $V^*$ is isomorphic to $\P^1$ punctured at $n$ points,
and a basis of $H^1$ of this hyperplane section is formed by small circles around the punctures. It is
clear with this basis that monodromy of hyperplane sections yields the defining representation of $S_n$
(i.e. $\Lambda \oplus \rm triv$). 
Since the trivial representation gives us a constant sheaf, we see that (up to constant complexes) the Radon transform $R$ sends 
${\bf R}_{\lambda}$ to $ {\bf L}_{\Lambda}$
(notice that $\Sec^1(X)=X$, $\lambda$ is the trivial representation of $S_1,$ and ${\mathcal R}_{\lambda}$
is the trivial local system). 
Using involutivity and the fact that $\Phi$ sends irreducible perverse sheaves to irreducible ones, we obtain $\Phi({\bf L}_{\Lambda})={\bf R}_{\lambda}$.

\end{example}

\subsection*{Main theorem for any curve}

 Let $C$ be a smooth projective curve of genus $g$. Let $C^{(n)}_{\rm dist}$ be the open subvariety
of $C^{(n)}$ (the symmetric power of $C$) consisting of the $n$-tuples of distinct points of $C$. For a 
representation $\Lambda$ of $S_n$, consider the local system ${\mathcal L}_{\Lambda}^{C, \rm big}$ 
on $C^{(n)}_{\rm dist}$ corresponding to the representation of $\pi_1(C^{(n)}_{\rm dist})$ which is the
pullback of $\Lambda$ via the natural map $\pi_1(C^{(n)}_{\rm dist}) \to S_n$ . 
Let $n>2g-2$ so that the classical Abel-Jacobi map $C^{(n)} \to {\rm Pic}^n(C)$ is a fibration.

{\indent} Restricting the local system ${\mathcal L}_{\Lambda}^{C, \rm big}$ 
to a fiber of the Abel-Jacobi map over a degree $n$ line bundle $M$ on $C$, we obtain a local system ${\mathcal L}^C_{\Lambda}$ 
on an open subset of the vector space $\P(H^0(M))$. From this local system we construct a perverse sheaf ${\bf L}^C_{\Lambda}$ on the vector space $H^0(M)$
as before by taking the intermediate extension of the lift of the local system ${\mathcal L}^C_{\Lambda}$ to the cone.
Denote by $\Sec^k(C)$
the secant variety of $C$ embedded into $\P(H^0(M)^*)$, and by ${\bf R}^C_{\lambda}$ the perverse
sheaf supported on the cone over $\Sec^k(C)$ constructed as before by intermediate extension of the lift
of ${\mathcal L}_{\lambda}^{C, \rm big}$ via the bundle map (see Lemma \ref{secantlemma2}(c)). Let $\lambda$ still denote the partition obtained
from $\Lambda$ by deleting its first row. We have a result similar to Theorem \ref{main}:

\begin{theorem} \label{abeljacobi}
If $|\lambda|\leq \frac{n+1}{2}-g$, then 
$$\Phi ({\bf L}^C_{\Lambda})={\bf R}^C_{\lambda}$$ 
\end{theorem}

\subsection*{Relation with M. Brion's result}

Let $\Lambda$ and $\Omega$ be two Young diagrams with $|\Lambda|=|\Omega|=n$,
$|\lambda|\leq \frac{n+1}{2}$ and $|\omega|\leq \frac{n+1}{2}$. 
 (For a partition denoted by a capital Greek letter, we still denote the partition obtained by deleting its first row by the corresponding
lowercase Greek letter). Recall that for a representation $\Sigma$ of $S_n$ we have the Kronecker 
coefficient $k_{\Lambda, \Omega}^{\Sigma}=[\Sigma:\Lambda \otimes \Omega]$; if $|\sigma|=|\lambda|+|\omega|$
we can also define the Littlewood-Richardson coefficient $c_{\lambda,\omega}^{\sigma}=[\sigma:{\rm Ind}_{S_{|\lambda|} \times S_{|\omega|}}^{S_{|\sigma|}} \lambda \boxtimes \omega]$. 
We prove:

\begin{proposition} \label{brion}   Suppose that $|\lambda|+|\omega| \leq n/2$. Then 

{\rm (a)} For any Young diagram $\Sigma$ with the first row strictly shorter than $n-|\lambda|-|\omega|$,
the Kronecker coefficient $m_{\Lambda, \Omega}^{\Sigma}$ is zero. 

{\rm (b)} For any Young diagram $\Sigma$
with the first row equal to $n-|\lambda|-|\omega|$, the Kronecker coefficient $k_{\Lambda, \Omega}^{\Sigma}$ is
equal to the Littlewood-Richardson coefficient $c_{\lambda, \omega}^{\sigma}$.
\end{proposition}

\begin{remark}
This result is the same as Corollary 2 in Section 3.4 of \cite{Brion} except for our additional assumption $|\lambda|+|\omega| \leq n/2$.
Without this assumption our method encounters some technical difficulties.
\end{remark}

\subsection*{Organization of the paper}

Section 2 is devoted to background on the Fourier-Deligne transform. 
In Section 3 we prove a number of lemmas about the Fourier-Deligne transform and apply them to prove Theorem \ref{main}. 
In Section 4 we generalize Theorem \ref{main} to local systems on fibers of the Abel-Jacobi map. 
In Section 5 we deduce from Theorem \ref{main} a part of  M. Brion's result as Proposition \ref{brion}.

\subsection*{Acknowledgements}
I am very grateful to Roman Bezrukavnikov for suggesting this problem and for many useful discussions and ideas. I am thankful to J.M. Landsberg for 
helping me locate the reference for M. Brion's result. I also thank A. Beilinson, S. Bloch, A. Libgober, L.-H. Lim, M. Mella, L. Oeding, R. Piene, and B. Sturmfels for useful conversations. 

\section{Fourier-Deligne transform}

The general reference for this section is the book \cite{KS}; see also the book \cite{KW} 
and the survey article \cite{I}, 
in which the theory is described in the $\ell$-adic setting, and the papers \cite{B}, \cite{BMV1}, \cite{BMV2},
in which the Fourier-Deligne transform is studied. 
Let $Y$ be an algebraic variety over $\C$ (or a real analytic stratified space). We work with the 
derived category of constructible sheaves $ D^b_c (Y)$. 
For a morphism of complex algebraic 
varieties (or stratified spaces) $f: Y \to Z$ we have the functors $ f_*, f_! : D^b_c (Y) \to D^b_c (Z)$ and
$ f^*, f^! : D^b_c (Z) \to D^b_c (Y)$ (note that we write $f_*$ for $R f_*$ and
 similarly for the other functors).

Perverse sheaves ${\rm Perv}(Y)$ form an abelian subcategory of $D^b_c (Y)$. We can form an 
 irreducible perverse sheaf, the IC extension ${\rm IC}(Y,\mathcal L)$ (also referred to as 
Goresky-MacPherson or intermediate extension of $\mathcal L$ to $Y$), for a local system $\mathcal L$ on 
a smooth locally closed subvariety $U \subset Y$. The category of local systems on a manifold $U$ is equivalent to the category
of representations of ${\pi}_1(U)$. We also have a classification theorem (see \cite{BBD}) which says that all irreducible 
perverse sheaves on $Y$ are of the form ${\rm IC}(Y,\mathcal L)$ for some $U$ and $\mathcal L$.

\begin{definition} \label{fourierdefn}
Let $V$ be a vector space over $\C$ and $p_1: V \times V^* \to V$, $p_2: V \times V^* \to V^*$
be the projections. Let $j: Q = \{(v, \xi) \ | \ {\rm Re} \langle v, \xi \rangle \leq 0 \} \hookrightarrow V \times V^*$
and ${\delta}_Q = j_* ({\C}_{Q})$. 
The Fourier-Deligne transform ${\rm \Phi: D^b_c (V) \to D^b_c (V^*)}$ is defined by
${\rm \Phi (\mathcal{F}) = (p_2)_! (p_1^*(\mathcal{F}) \otimes {\delta}_Q)}$.
\end{definition}

Note that sometimes we refer to $\Phi$ simply as "the Fourier transform."

\begin{remark} (a) $\Phi$  is defined using a stratification of $V \times V^*$ which is only real analytic,
but if $\mathcal{F}$ is constructible with respect to a complex analytic stratification on $V$, $\Phi(\F)$
is likewise on $V^*$.

(b)  $\Phi$ sends the category $D^b_{c, \C^*} (V)$ of monodromic (i.e. ${\C}^*$-equivariant) objects 
on $V$ to $D^b_{c, \C^*} (V^*)$; in this text we always
 work with $\Phi : D^b_{c, \C^*} (V) \to D^b_{c, \C^*} (V^*)$.

(c) $\Phi$ sends perverse sheaves to perverse sheaves and induces an exact functor between the abelian 
categories of perverse sheaves on $V$ and $V^*$.
\end{remark}

\begin{remark} $\Phi$  can be defined for $\ell$-adic sheaves  
on $V$ over ${\mathbb F}_q$; then it is
compatible (under the Frobenius trace) with the Fourier transform for functions on $V$ (see \cite{I}).
\end{remark}

\begin{remark}

The Weyl algebra $\mathfrak W$ is the free algebra $\C \langle x_0,...x_n,{\partial}_0,...,{\partial}_n \rangle$
modulo the relations $x_i {\partial}_j = {\partial}_j x_i$ for $i \neq j$ and ${\partial}_i x_i - x_i {\partial}_i =1$.
The regular holonomic D-modules ${\rm {D-mod}^{r.h.}}(V)$ on $V$ form a subcategory of the 
category of modules over $\mathfrak W$.
The Riemann-Hilbert correspondence ${\rm RH}: {\rm {D-mod}^{r.h.}}(V) \to {\rm Perv}(V) $
is an equivalence of categories. Under $\rm RH$ the Fourier-Deligne transform defined above is compatible with 
the Fourier transform $F$ for $\mathfrak W$-modules given as follows. If $M$ is a $\mathfrak W$-module, $F(M)=M$ as a vector space and the new action of $\mathfrak W$ on $m \in M$ is $x_i^{\rm new}(m)={\partial}_i(m)$, ${\partial}_i^{\rm new}(m)= - x_i(m)$

\end{remark}

In the course of proving our main result we also need the following more general definition 
of the (relative) Fourier transform:

\begin{definition} Let $E$ be a complex vector bundle over a complex algebraic variety $Y$, $E^*$ the dual bundle, and 
$p_1: E \times_Y E^* \to E$, $p_2: E \times_Y E^* \to E^*$
be the projections. Let $j: Q = \{((x, v), (x, \xi)) \ | \ {\rm Re} \langle v, \xi \rangle \leq 0 \} \hookrightarrow E \times_Y E^*$
and ${\delta}_Q = j_* ({\C}_{Q})$. 
The Fourier-Deligne transform ${\rm \Phi_E: D^b_c (E) \to D^b_c (E^*)}$ is defined by
${\rm \Phi_E (\mathcal{F}) = (p_2)_! (p_1^*(\mathcal{F}) \otimes {\delta}_Q)}$.
\end{definition}

We note that the remarks above apply here, and we will work with $\Phi_E : D^b_{c, \C^*} (E) \to D^b_{c, \C^*} (E^*)$,
where $\C^*$ acts in the fibers of $E$ and $E^*$.

\section{Proof of the main theorem for $\P^1$}

Let $W \cong \C^2$ be a two-dimensional complex vector space. 
Consider the bilinear pairing $\omega: {\Sym}^k(W) \times {\Sym}^{n-k}(W) \to {\Sym}^n(W)$
which corresponds to multiplication of polynomials.
The linear maps $\omega(v,-)$, obtained by fixing the first coordinate in the pairing $\omega$
to be $v$, combine for all $v$ to give a map of vector bundles over ${\P}^k=\P({\Sym}^k(W))$, namely 
$ i : \mathcal{E}= \O_{\P^k}(-1) \otimes {\Sym}^{n-k}(W) \to \O_{\P^k} \otimes {\Sym}^n(W)$.
We denote the bundle map for $\mathcal{E}$ by $p_{\mathcal{E}}: \mathcal{E} \to {\P}^k$.
Note that the map $i$ makes $\mathcal{E}$ a subbundle of the trivial bundle 
$\O_{\P^k} \otimes {\Sym}^n(W)$, so the total space of $\mathcal{E}$ embeds into
${\P}^k \times {\Sym}^n(W)$. Hence via the second projection we get a map
 $\pi : \mathcal{E} \to {\Sym}^n(W)$, $\pi = p_2 \circ i$. (Note that we sometimes use the same notation 
for a vector bundle and its total space as we did for $\mathcal{E}$.) See the diagram below.

\begin{tikzpicture}
  \matrix (m) [matrix of math nodes,row sep=3em,column sep=3em,minimum width=2em, ampersand replacement=\&]
  {
     \mathcal{E} \& {\P}^k \times {\rm Sym}^n(W) \& {\rm Sym}^n(W), \& \pi = p_2 \circ i\\
     \ \&  {\P}^k \& \ \\};
  \path[-stealth]
    (m-1-1) 
          edge node [above] {$i$} (m-1-2)
          edge node [above] {$p_{\mathcal{E}}$} (m-2-2)

    (m-1-2) edge node [right] {$p_1$} (m-2-2)         
            edge node [above] {$p_2$} (m-1-3);
\end{tikzpicture}

Recall that for a positive integer $r$ and a representaion $\rho$ of $S_r$ we introduced 
the local system ${\mathcal L}_{\rho}$ on the open set ${\P}^r - D_r$ in ${\P}^r= \P(\Sym^r(W))$ where $D_r$
is the discriminant locus. We denote by ${\widetilde L}_{\rho}$ the intermediate extension of ${\mathcal L}_{\rho}$
to ${\P}^r$ and by ${\bf L}_{\rho}$ the intermediate extension to $\Sym^r(W)$ of the lift
of ${\mathcal L}_{\rho}$ to the complement of the cone over $D_r$ in $\Sym^r(W)$.

\begin{lemma}\label{Ind} 
Let $\lambda$ be a representation of $S_k$ and let $I$ denote the trivial representation of $S_{n-k}$. 
Consider the induced
representation of $S_n$ given by $\bar{\Lambda}={\rm Ind}^{S_n}_{S_k \times S_{n-k}} (\lambda \boxtimes I)$.
Then, with the above notation, we have ${\bf L}_{\bar{\Lambda}} = {\pi}_*  p_{\mathcal{E}}^*  {\widetilde L}_{\lambda}$.

\end{lemma}

\begin{proof}

For any natural number $d$ consider the map ${\phi}_d : ({\P}^1)^d \to ({\P}^1)^{(d)} = {\P}^d $. Notice that
for any representation $\rho$ of $S_d$ we have that ${\widetilde L}_{\rho} = \rm{Hom}_{S_d} (\rho, ({\phi}_d)_* \underline{\C}[d] )$
(if $\rho$ is an irreducible representation then ${\widetilde L}_{\rho}$ is the isotypic component for the action of $S_d$ on the sheaf $({\phi}_d)_* \underline{\C}[d]$). Also let 
$m=m_{k , n-k}: {\P}^k \times {\P}^{n-k} \to {\P}^n$ denote the map given by multiplication of polynomials.

Notice that 
$\phi = {\phi}_n$ factors through ${\P}^k \times {\P}^{n-k}$ as
${\phi} = m \circ ({\phi}_k \times {\phi}_{n-k})$, as depicted in the following diagram:

\begin{tikzpicture} \label{factoring}
  \matrix (m) [matrix of math nodes,row sep=3em,column sep=6em,minimum width=2em, ampersand replacement=\&]
  {
     ({\P}^1)^n \& {\P}^k \times {\P}^{n-k} \\
     \ \&  {\P}^n \\};
  \path[-stealth]
    (m-1-1) 
          edge node [above] {${\phi}_k \times {\phi}_{n-k}$} (m-1-2)
          edge node [above] {$\phi$} (m-2-2)

    (m-1-2) edge node [right] {$m$} (m-2-2);

\end{tikzpicture}

Further, for any representations $\mu$ of $S_k$ and $\nu$ of $S_{n-k}$
we have ${\widetilde L}_{\mu} \boxtimes {\widetilde L}_{\nu} = \rm{Hom}_{S_k \times S_{n-k}} (\mu \boxtimes \nu, ({\phi}_k \times {\phi}_{n-k})_* \underline{\C}[n] )$ on ${\P}^k \times {\P}^{n-k}$. 

We have the following consequence of these observations: 
$$m_*  ({\widetilde L}_{\mu} \boxtimes {\widetilde L}_{\nu})= 
m_* \rm{Hom}_{S_k \times S_{n-k}}(\mu \boxtimes \nu, ({\phi}_k \times {\phi}_{n-k})_* \underline{\C}[n]) = $$
$$\rm{Hom}_{S_k \times S_{n-k}}(\mu \boxtimes \nu, {\phi}_* \underline{\C}[n]) =
\rm{Hom}_{S_n}(\rm{Ind}_{S_k \times S_{n-k}}^{S_n} (\mu \boxtimes \nu), {\phi}_* \underline{\C}[n])$$

Note that when we delete the zero section of the bundle $\mathcal{E}$ 
and take the quotient by the ${\C}^*$-action 
(i.e. pass to the projectivization of the vector bundle $\mathcal{E}$),
the map $\pi$ is replaced by the map $m$ above.  
Hence we obtain the statement of the lemma from
the above chain of equalities by taking $\mu=\lambda$ and $\nu = I$.

\end{proof}

 The following two lemmas on the Fourier transform are proved in \cite{KW} as 
Corollary III.13.4 and Corollary III.13.3.

\begin{lemma}
Let $L$ be a constructible complex on a complex algebraic variety $A$. Let ${F}$
be a subbundle of the trivial bundle $A \times W$ where $W$ is a vector space,
with bundle map $p_{{F}} : {F} \to A$ and inclusion $i: {F} \to A \times W$.
Let ${{F}}^{\perp}$ be the orthogonal subbundle of $A \times W^*$,
with bundle map $p_{{{F}}^{\perp}} : {{F}}^{\perp} \to A$
and inclusion $i^{\perp}: {{F}}^{\perp} \to A \times W^*$.
Then 
$$\Phi_{A \times W} ( i_* p_{{F}}^* L) =  i^{\perp}_* p_{{{F}}^{\perp}}^* L$$
\end{lemma}

\begin{lemma} \label{Dir}
Let $q: S \to B$ be a proper map of complex algebraic varieties,  let ${F}$ be a vector 
bundle on $B$, let $q^* {F} $ be the pullback bundle on $S$, and 
$Q: q^* {F} \to {F}$ the map of total spaces of bundles.
Let $R:  q^* {F}^* \to {F}^* $ be the corresponding map for
the total space of the dual bundle. Then for a monodromic constructible complex $M$
on the bundle $q^* {F}$ we have
$$\Phi_{{F}} (Q_*  M) = R_*( \Phi_{q^* {F}} M)$$

 \end{lemma}

We let ${\mathcal{E}}^{\perp}$ be the subbundle of the
 trivial bundle $\P^k \times {\rm Sym}^n (W^*)$ over ${\P}^k$ which is orthogonal
 to the subbundle $\mathcal{E}$ of $\P^k \times {\Sym}^n (W)$.
Let $p_{{\mathcal{E}}^{\perp}} : {\mathcal{E}}^{\perp} \to {\P}^k$
be the bundle map, and let ${\pi}^{\perp}= p_2^{\perp} \circ i^{\perp}$  be the composition of 
the inclusion $i^{\perp}$ of ${\mathcal{E}}^{\perp}$ into $\P^k \times {\rm Sym}^n (W^*)$
and the second projection $p_2^{\perp}$. See the following diagram.

\begin{tikzpicture}
  \matrix (m) [matrix of math nodes,row sep=3em,column sep=3em,minimum width=2em, ampersand replacement=\&]
  {
     {\mathcal E}^{\perp} \& {\P}^k \times {\rm Sym}^n(W^*) \& {\rm Sym}^n(W^*), \& \pi^{\perp} = p_2^{\perp} \circ i^{\perp}\\
     \ \&  {\P}^k \& \ \\};
  \path[-stealth]
    (m-1-1) 
          edge node [above] {$i^{\perp}$} (m-1-2)
          edge node [above] {$p_{{\mathcal E}^{\perp}}$} (m-2-2)

    (m-1-2) edge node [right] {$p_1^{\perp}$} (m-2-2)
            edge node [above] {$p_2^{\perp}$} (m-1-3);
\end{tikzpicture}

\

Combining the results of the last three lemmas, we obtain 

\begin{proposition} \label{coroflemmas} Using the notation of Lemma \ref{Ind}, we have
$$\Phi ({\bf L}_{\bar{\Lambda}}) ={\pi}^{\perp}_* p_{{\mathcal{E}}^{\perp}}^* {\widetilde L}_{\lambda}$$
\end{proposition}

\begin{proof} Note that we apply Lemma \ref{Dir} to the proper map ${\P}^k \to {\rm point}$.
\end{proof}

\begin{lemma} \label{secantlemma}

{\rm (a)} The image of ${\mathcal E}^{\perp}$ under $\pi^{\perp}$ is the cone over ${\Sec^k}$ in $\Sym^n(W^*)$.

{\rm (b)} If $k \leq \frac{n+1}{2}$, the map $\pi^{\perp}$ is an isomorphism restricted to $({\pi}^{\perp})^{-1}(U_k)$ where $U_k$ is the cone (with vertex removed) over ${\Sec^k} - \Sec^{k-1}$ in $\Sym^n(W^*)$.

{\rm (c)} If $k \leq \frac{n+1}{2}$, the map $p_{{\mathcal E}^{\perp}}$ restricted to $({\pi}^{\perp})^{-1}(U_k)$ is a fiber bundle over the configuration space $\Conf_k(X) = \P^k - D_k$.

\end{lemma} 

\begin{proof} 

To prove (a) we notice that the fiber of ${\mathcal E}^{\perp}$
over $q \in (\P^1)^{(k)}$ can be thought of as the space of homogeneous polynomials $p$ of degree $n$ 
which are annihilated by $q(\partial_x, \partial_y)$. On the other hand, Lemma 4 in \cite{CS} 
(see also \cite{L}, Section 3.5.3) says that a point $p \in \Sym^n(W^*)$ is in $\Sec^k$ if and only there exists
a polynomial $q(\partial_x,\partial_y)$ which annihilates $p(x,y)$.
 (Note that Lemma 4 is stated in \cite{CS} only for $k \leq \frac{n}{2}$ 
but for $k=\frac{n+1}{2}$, if it is an integer, the statement of Lemma 4 in \cite{CS} holds trivially).

Parts (b) and (c) follow from the fact that (for $k \leq \frac{n+1}{2}$) two different $k$-secant planes 
of $X$ meet at points of the $(k-1)$-th 
secant variety (i.e. there is only one $k$-secant plane that passes through a point of 
$\Sec^k - \Sec^{k-1}$), which follows by calculating a Vandermonde determinant.
A more general version of this argument appears in the proof of Lemma \ref{secantlemma2}.

\end{proof}

\begin{remark} The fact that ${\rm Sec}^k  - {\rm Sec}^{k-1}$ 
is smooth (which follows from part (b) of Lemma \ref{secantlemma}) is 
a special case of \cite{Bertram},  Corollary 1.6.
\end{remark} 

\begin{remark} \label{eperpisabssec}
By construction the fiber of ${\mathcal E}^{\perp}$ over a $k$-tuple
of points in $\P^k=X^{(k)}$ is the $k$-plane in $\Sym^n(W^*)$ passing
 through these $k$ points and the point $0$ (if the points in the $k$-tuple
are not distinct, the plane becomes an osculating plane to $X$). Hence the
projectivization of the bundle ${\mathcal E}^{\perp}$ is the $k$-th "secant bundle"
of the curve $X$ described in \cite{Bertram}, Section 1.
\end{remark}

Let $\rho$ be a partition of $k$. The fiber bundle $p_{{\mathcal E}^{\perp}} \circ (\pi^{\perp})^{-1}: U_k \to \Conf_k(X)$ 
yields a map of groups $\pi_1(U_k) \to \pi_1(\Conf_k(X))=B_k(\P^1) \to S_k$. 
This allows us to make the following

\begin{definition}\label{Rdefn}
The local system ${\mathcal R}_{\rho}$ on $U_k$ is defined by lifting the representation $\rho$ via the above map
$\pi_1(U_k) \to S_k$. We define a perverse sheaf ${\bf R}_{\rho}$ as the intermediate extension of ${\mathcal R}_{\rho}$
to $\Sym^n(W^*)$. 
\end{definition}

Now we are ready to prove Theorem \ref{main}:

\begin{proof}[Proof of Theorem \ref{main}]

We prove the statement by decreasing induction on the length of the first row of $\Lambda$
(or, equivalently, by increasing induction on $|\lambda|$).
The base of induction is furnished by Example \ref{regrep}, where the length of the first row
of $\Lambda$ is $|\Lambda|-1$. 

Proposition \ref{coroflemmas} shows that 
$\Phi ({\bf L}_{\bar{\Lambda}}) ={\pi}^{\perp}_* p_{{\mathcal{E}}^{\perp}}^* {\widetilde L}_{\lambda}$.
By Pieri's rule the irreducible decomposition of the induced representation 
$\bar{\Lambda}={\rm Ind}_{S_k \times S_{n-k}}^{S_n} (\lambda \boxtimes I)$ of $S_n$
has exactly one irreducible representation with the first row of length at most $n-k$, 
 this representation is $\Lambda$, and it has multiplicity $1$. Hence 
${\bf L}_{\bar{\Lambda}} ={\bf L}_{\Lambda} \oplus \bigoplus_{\Omega, |\omega|< k} a_{\Omega} {\bf L}_{\Omega}$, where $a_{\Omega}$ is the multiplicity of $\Omega$ in $\bar{\Lambda}$.
The inductive hypothesis implies that for all $\Omega$ with $|\omega|<k$ we have 
$\Phi({\bf L}_{\Omega})={\bf R}_{\omega}$. Hence 
$\Phi ({\bf L}_{\bar{\Lambda}}) =
\Phi({\bf L}_{\Lambda}) \oplus \bigoplus_{\Omega, |\omega|< k} a_{\Omega} {\bf R}_{\omega}$.
Since by definition ${\bf R}_{\omega}$ is an irreducible perverse sheaf supported on the cone over  
$\Sec^{|\omega|} \subset \Sec^{k-1}$, the restriction of $\Phi({\bf L}_{\Lambda})$ to $U_k$
coincides with the restriction to $U_k$ of 
$\Phi ({\bf L}_{\bar{\Lambda}}) ={\pi}^{\perp}_* p_{{\mathcal{E}}^{\perp}}^* {\widetilde L}_{\lambda}$,
which by definition is ${\mathcal R}_{\lambda}$. Note that by Lemma \ref{secantlemma} (a) 
${\pi}^{\perp}_* p_{{\mathcal{E}}^{\perp}}^* {\widetilde L}_{\lambda}$ (and hence $\Phi({\bf L}_{\Lambda})$)
is supported on $\Sec^k$. Since Fourier transform of an irreducible perverse sheaf is an irreducible perverse
sheaf, by the structure theorem of \cite{BBD} $\Phi({\bf L}_{\Lambda})$ must be the intermediate 
extension of ${\mathcal R}_{\lambda}$, that is ${\bf R}_{\lambda}$.

\end{proof}

The following is a corollary of the above proof:

\begin{corollary}
Even if  the condition $|\lambda|\leq \frac{|\Lambda|+1}{2}$ does not hold,
$\Phi ({\bf L}_{\Lambda})$ is a direct summand in 
the semisimple 
perverse sheaf
${\pi}^{\perp}_* p_{{\mathcal{E}}^{\perp}}^* {\widetilde L}_{\lambda}$.
\end{corollary}

\section{Main theorem for any curve}

In this section we prove Theorem \ref{abeljacobi} which generalizes Theorem \ref{main},
replacing the rational normal curve $X$ by an arbitrary 
smooth projective curve $C$ of genus $g$ embedded into $\P^n$. We start with a lemma:

\begin{lemma} \label{separatepoints}
 For a very ample line bundle $L$ on a smooth projective curve $C$, consider the embedding of $C$ 
into the projective space ${\P}(H^0(C,L)^*)$. Then the following two conditions are equivalent:

(a) 
Two $k$-secant planes of $C$ one of which contains the points $x_1,...,x_k \in C$ and
the other one -- $y_1,...,y_k \in C$ (such that $x_1,...,x_k,y_1,...,y_k$ are all distinct)
do not intersect;

(b) $h^0(L)-h^0(L(-D))=2k$ for any effective divisor $D$ of degree $2k$
(this condition is known as "$D$ separates $2k$ points").

\end{lemma} 

\begin{proof}

First note that the projectivized span of all points of an effective divisor $D$ is given by
 $\P({\rm Ker} (H^0(L)^* \to H^0(L(-D))^*))$. If $D$ has degree $2k$, the long
exact sequence of cohomology associated to the short exact sequence $0 \to L(-D) \to L \to L_D \to 0$
shows that the points of $D$ span a projectivized linear space of expected dimension $2k-1$ 
if and only if $h^0(L)-h^0(L(-D))=2k$. Finally, two $k$-secant planes in (a) intersect if and 
only if the projectivized linear space spanned by the points of $D=\{x_1,...,x_k,y_1,...,y_k\}$
has dimension less than $2k-1$. 
\end{proof}

To prove the Theorem \ref{abeljacobi}, we generalize the constructions of the previous section.
Recall that we have a smooth projective curve $C$ of genus $g$ and a line bundle $M$ of degree $n$ on $C$.
We construct a vector bundle ${\mathcal E}_C$ over $C^{(k)}$ (analogous to the bundle $\mathcal E$ above)
as a subbundle of the trivial bundle $C^{(k)} \times H^0(C,M)$  by requiring the fiber of $\mathcal E_C$ over a point $D \in C^{(k)}$ to be given by ${\rm Image}(H^0(C,M(-D)) \to H^0(C,M))$. (Note that ${\mathcal E}_C$ is a vector bundle for
$k \leq n+1 -2g$ since then $H^1(C,M(-D))=0$ by Serre duality.) See the diagram below.

\begin{tikzpicture}
  \matrix (m) [matrix of math nodes,row sep=3em,column sep=3em,minimum width=2em, ampersand replacement=\&]
  {
     {\mathcal E}_C \& C^{(k)} \times H^0(C,M) \& H^0(C,M), \& \pi_C = p_{2,C} \circ i_C\\
     \ \&  C^{(k)} \& \ \\};
  \path[-stealth]
    (m-1-1) 
          edge node [above] {$i_C$} (m-1-2)
          edge node [above] {$p_{{\mathcal E}_C}$} (m-2-2)
    
    (m-1-2) edge node [right] {$p_{1,C}$} (m-2-2)
            edge node [above] {$p_{2,C}$} (m-1-3);
\end{tikzpicture}

Recall that for a representation $\rho$ of $S_r$ we introduced the local system ${\mathcal L}_{\rho}^{C, \rm big}$
on the open subvariety $C^{(r)}_{\rm dist}$ of the symmetric power $C^{(r)}$. We denote its intermediate 
extension to $C^{(r)}$ by  $ L_{\rho}^{C,\rm big}$. We also denote the intermediate extension of
the local system ${\mathcal L}^C_{\rho}$ to $\P(H^0(C,M))$ by ${\widetilde L}^C_{\rho}$ and 
the intermediate extension to $H^0(C,M)$ of the lift of ${\mathcal L}^C_{\rho}$ to the cone by ${\bf L}^C_{\rho}$. 
The following generalizes Lemma \ref{Ind}:

\begin{lemma} \label{Ind2}
Let $\lambda$ be a representation of $S_k$ and let $I$ denote the trivial representation of $S_{n-k}$. 
Let $\bar{\Lambda}={\rm Ind}^{S_n}_{S_k \times S_{n-k}} (\lambda \boxtimes I)$.
Then 
${\bf L}^C_{\bar{\Lambda}} = (\pi_C)_*  (p_{{\mathcal E}_C})^* L_{\lambda}^{C, \rm big}$.
\end{lemma}

\begin{proof}
The argument parallels the proof of Lemma \ref{Ind} with several adjustments. Namely, 
in place of the diagram in the proof of Lemma \ref{Ind} we use the diagram below.
In this diagram, $\phi_r^C: C^r \to C^{(r)}$ is the natural map for any $r$, $p_1^{\prime}$
is the first projection, and $A_M$ is the fiber of the Abel-Jacobi map over $M$. 
In addition to these, $i_C^{\prime}$ is
the map derived from $i_C$ by embedding the fibers $H^0(C,M(-D))$ of ${\mathcal E}_C$
 into $C^{(n-k)}$ as fibers ot the Abel-Jacobi map, and $p^{\prime}_{{\mathcal E}_C}$ and $\pi^{\prime}_C$
are the analogs of $p_{{\mathcal E}_C}$ and $\pi_C$ for the projectivization $\P({\mathcal E}_C)$ of the bundle 
${\mathcal E}_C$.

\begin{tikzpicture} \label{factoring}
  \matrix (m) [matrix of math nodes,row sep=3em,column sep=6em,minimum width=2em, ampersand replacement=\&]
  {
     \         \&                C^{(k)}             \&       \                            \\
     C^n   \& C^{(k)} \times C^{(n-k)} \&  \P({\mathcal E}_C)    \\
     \         \&                C^{(n)}             \&   \P(H^0(C,M))           \\};
  \path[-stealth]
    (m-2-1) 
          edge node [above] {${\phi}_{k}^{C} \times {\phi}_{n-k}^C$} (m-2-2)
          edge node [above] {$\phi_n^C$} (m-3-2)

    (m-2-2) edge node [right] {$m_C$} (m-3-2)
                 edge node [right] {$p_1^{\prime}$} (m-1-2)

    (m-2-3) edge node [above] {$p_{{\mathcal E}_C}^{\prime}$} (m-1-2)
                 edge node [below] {$i_C^{\prime}$}                            (m-2-2)
                 edge node [right] {$\pi_C^{\prime}$}                           (m-3-3)

    (m-3-3) edge node [above] {$A_M$} (m-3-2);
\end{tikzpicture}

The argument of Lemma \ref{Ind} gives 
$L_{\bar \Lambda}^{C, \rm big} = (m_C)_* (L_{\lambda}^{C, \rm big} \boxtimes \underline{\C})$. 
Hence
$${\widetilde L}_{\bar \Lambda}^{C}=A_M^* L_{\bar \Lambda}^{C, \rm big}
= (\pi_C^{\prime})_* (i_C^{\prime})^*(L_{\lambda}^{C, \rm big} \boxtimes \underline{\C})
=(\pi_C^{\prime})_*  (p^{\prime}_{{\mathcal E}_C})^*  L_{\lambda}^{C, \rm big}$$
From this the statement of the Lemma follows. 
(Note that we used proper base change in the second equality above).

\end{proof}

We also have to prove a generalization of Lemma \ref{secantlemma}:

\begin{lemma} \label{secantlemma2}

{\rm (a)} The image of ${\mathcal E}_C^{\perp}$ under $\pi_C^{\perp}$ is the cone over ${\Sec^k}(C)$ in $H^0(C,M)^*$.

{\rm (b)} If $k \leq \lfloor \frac{n+1}{2} \rfloor-g$, the map $\pi_C^{\perp}$ is an isomorphism restricted to $({\pi}_C^{\perp})^{-1}(U^C_k)$ where $U^C_k$ is the cone (with vertex removed) over ${\Sec^k}(C) - \Sec^{k-1}(C)$ in $H^0(C,M)^*$.

{\rm (c)} If $k \leq \lfloor \frac{n+1}{2} \rfloor-g$, the map $p_{{\mathcal E}_C^{\perp}}$ restricted to $({\pi}_C^{\perp})^{-1}(U^C_k)$ is a fiber bundle over $C^{(k)}_{\rm dist}$.
\end{lemma}

\begin{proof} To prove (a), we notice that the fiber of ${\mathcal E}_C$ over $D \in C^{(k)}$ is given by 
${\rm Image}(H^0(M(-D)) \to H^0(M))$, which is othogonal to the cone over the $k$-secant plane 
passing through the points of $D$ given by ${\rm Ker} (H^0(M)^* \to H^0(M(-D))^*)$. 

Parts (b) and (c) will follow if we show that the condition of part (a) of Lemma \ref{separatepoints} holds.
By Lemma \ref{separatepoints} it is enough to show that part (b) of Lemma \ref{separatepoints} holds for
$k=\lfloor \frac{n+1}{2} \rfloor-g$. That is, we would like to show that $h^0(C,M)=h^0(C,M(-D))+2k$ holds for 
any effective $D$ of degree $2k$.

By Riemann-Roch we have $h^0(C,M) - h^1(C,M)=n+1-g$ and $h^0(C,M(-D)) - h^1(C,M(-D))=n-2k+1-g$
for an effective $D$ with $|D|=2k$. We have $h^1(M)=h^0(K \otimes M^{-1})$. We also have 
${\rm deg} K =2g-2$ so ${\rm deg} (K \otimes M^{-1})=2g-2-n<0$ because we assumed that $n>2g-2$.
Hence $h^1(C,M)=0$ and $h^0(C,M)=n+1-g$. Subtracting, we get 
$h^0(C,M)-h^0(C,M(-D))=h^1(C,M(-D)) + 2k$. So to ensure that the condition of Lemma \ref{separatepoints} $(a)$
holds it is enough to have $h^1( M(-D))=0$ for all $D$ of degree $2k$. In other words, it is enough to
make sure that $h^0(K \otimes M^{-1}(D))=0$ for all $D$ of degree $2k$. For this it is enough to have
${\rm deg}(K \otimes M^{-1}(D)) =0$, that is $2g-2-n+2k<0$. The last inequality holds for 
$k=\lfloor \frac{n+1}{2} \rfloor-g$.
\end{proof}

Using Lemmas \ref{Ind2} and \ref{secantlemma2} in place of Lemmas \ref{Ind} and \ref{secantlemma}, the proof of Theorem \ref{abeljacobi}
 is identical to the proof of Theorem \ref{main}.

\section{Relation with M. Brion's result}

In this section $\otimes$ stands for the left derived tensor product of objects of $D^b_c(Y)$, where $Y$ is an algebraic variety.
The following lemma will be used in this section:

\begin{lemma} \label{somesheafisperverse} Let $f: Y \to Z$ be a finite morphism with $Y$, $Z$ smooth varieties of 
dimension $n$. Then $f_* (\underline{\C}[n]) \otimes f_* (\underline{\C}[n])$ is a perverse sheaf on $Z$.

\end{lemma}

\begin{proof}

Using base change we can write 
$f_* (\underline{\C}[n]) \otimes f_* (\underline{\C}[n]) = g_* (\underline{\C}[n])$
 where $g: Y{\times}_Z Y \to Z$. 

\begin{tikzpicture}
  \matrix (m) [matrix of math nodes,row sep=3em,column sep=4em,minimum width=2em]
  {
     Y {\times}_Z Y & Y \times Y \\
     Z & Z \times Z \\};
  \path[-stealth]
    (m-1-1) edge node [left] {$g$} (m-2-1)
            edge (m-1-2)
    (m-2-1) edge node [above] {$\Delta$} (m-2-2)
    (m-1-2) edge node [right] {$f \times f$}(m-2-2);
          
\end{tikzpicture}

Notice that $Y{\times}_Z Y$
is locally a complete intersection because locally $Y \cong {\C}^n$ and $Z \cong {\C}^n$
and $f=(f_1,...,f_n): {\C}^n \to {\C}^n$
and hence $Y {\times}_Z Y \subset Y \times Y$ is of dimension $n$ and 
given by $n$ equations $f_i(x^{\prime})=f_i(x^{\prime \prime})$ 
where $(x^{\prime}, x^{\prime \prime}) \in Y \times Y$.
By a theorem of Hamm and Le \cite{HL}, Cor. 3.2.2,
the constant sheaf on a locally complete intersection is perverse; hence 
$\underline{\C}[n] \in D_c^b (Y {\times}_Z Y)$ is perverse. 
Since $g$ is a finite map, $g_* (\underline{\C}[n])$ is also perverse.

\end{proof}

From Lemma \ref{somesheafisperverse} we deduce the following

\begin{corollary} \label{tensorproductisperverse}
Let $\Lambda$, $\Omega$, and $\Sigma$ be partitions of $n$ and let $V=\Sym^n(W)$. Then

{\rm (a)} The object ${\bf L}_{\Lambda} \otimes {\bf L}_{\Omega} \in D^b_c(V)$ is a perverse sheaf.

{\rm (b)} The multiplicity of the subquotient ${\bf L}_{\Sigma}$ in ${\bf L}_{\Lambda} \otimes {\bf L}_{\Omega}$
is $k^{\Sigma}_{\Lambda, \Omega}$.
\end{corollary}

\begin{proof} To prove (a), recall from the proof of Lemma \ref{Ind} the map $\phi=\phi_n$ and the fact that
${\widetilde L}_{\rho}$ is a direct summand in $\phi_* (\underline{\C}[n])$ for a representation $\rho$ of $S_n$. 
Hence ${\widetilde L}_{\Lambda} \otimes {\widetilde L}_{\Omega}$ is a direct summand in $\phi_* (\underline{\C}[n]) \otimes \phi_* (\underline{\C}[n])$
which is a perverse sheaf by Lemma \ref{somesheafisperverse}. Hence ${\widetilde L}_{\Lambda} \otimes {\widetilde L}_{\Omega}$
is a perverse sheaf. It follows that the same is true for ${\bf L}_{\Lambda} \otimes {\bf L}_{\Omega}$.
Part (b) follows  by the classification theorem of perverse sheaves
from the fact that on the open subvariety $\P^n-D_n$ we have ${\mathcal L}_{\Lambda} \otimes {\mathcal L}_{\Omega} = {\mathcal L}_{\Lambda \otimes \Omega}$.

\end{proof}

To prove Proposition \ref{brion}, we will use the interaction of the Fourier transform 
with convolution, which is defined as follows:

\begin{definition}
Let $V$ be a vector space and $s: V\times V \to V$ be the sum map
and let $\mathcal F,  \mathcal G \in D_c^b(V)$. Then the convolution 
of $\mathcal F$ and $\mathcal G$ is the object of $D_c^b(V)$ given by
$\mathcal{F}* \mathcal{G}={s}_! (\mathcal{F} \boxtimes \mathcal{G})$.
\end{definition}

\begin{remark} The Fourier-Deligne transform sends the (derived) tensor product of 
$\mathcal F, \mathcal G \in D_c^b(V)$ to the convolution $\mathcal F * \mathcal G$.
\end{remark}

\begin{remark} \label{convolutionsupportedonjoin}
Let $\mathcal{F}, \mathcal{G} \in D^b_c(V)$ be two objects such that $\mathcal{F}$ is supported on a closed subvariety $Y$ of $V$
and $\mathcal{G}$ is supported on a closed subvariety $Z$ of $V$. Then the convolution ${\mathcal F} * {\mathcal G}$ is supported on the join $J(Y,Z)$ of $Y$ and $Z$
(where $J(Y,Z)$ is defined as the closure of the union of lines that pass through $y,z$ where $y \in Y$,
$z \in Z$ and $y \neq z$).

\end{remark}

We finish with the proof of a part of M. Brion's result.

\begin{proof}[Proof of Proposition \ref{brion}] To prove part (a), assume on the contrary that the Kronecker coefficient 
$k_{\Lambda, \Omega}^{\Sigma}$ is nonzero. Then by Corollary \ref{tensorproductisperverse} the perverse sheaf
${\bf L}_{\Lambda} \otimes {\bf L}_{\Omega}$ has a subquotient isomorphic to ${\bf L}_{\Sigma}$. 
By exactness of the Fourier transform, it follows that the perverse sheaf ${\bf R}_{\lambda} * {\bf R}_{\omega}$ has
a subquotient isomorphic to ${\bf R}_{\sigma}$;
hence its restriction to an open subset of $\Sec^{|\sigma|}$ is a nonzero local system. 
However, by Remark \ref{convolutionsupportedonjoin} the convolution ${\bf R}_{\lambda} * {\bf R}_{\omega}$
is supported on the join of $\Sec^{|\lambda|}$ and $\Sec^{|\omega|}$ 
which is $\Sec^{|\lambda|+|\omega|}$, a closed subset of $\Sec^{|\sigma|}$ -- a contradiction.

To prove (b), we use the diagram below.

\begin{tikzpicture}
  \matrix (m) [matrix of math nodes,row sep=3em,column sep=4em,minimum width=2em]
  {
    V^* \times V^* & V^* \\
   {\mathcal E}^{\perp}_{|\lambda|} \times {\mathcal E}^{\perp}_{|\omega|} & {\mathcal E}^{\perp}_{|\sigma|} \\
   \P^{|\lambda|} \times \P^{|\omega|} & \P^{|\sigma|} \\};
  \path[-stealth]
    (m-1-1) 
            edge node [above] {$s$} (m-1-2)
    (m-2-1) edge node [above] {$F$} (m-2-2)
            edge node [left] {$\pi^{\perp} \times \pi^{\perp}$} (m-1-1)
            edge node [left] {$p_{{\mathcal E}^{\perp}} \times p_{{\mathcal E}^{\perp}}$} (m-3-1)
    (m-2-2) edge node [right] {$\pi^{\perp}$}(m-1-2)
                 edge node [right] {$p_{{\mathcal E}^{\perp}}$}(m-3-2)
    (m-3-1) edge node [above] {$m$} (m-3-2);    
\end{tikzpicture}

 In this diagram ${\mathcal E}^{\perp}_k$ is the bundle introduced in Section 3
for a specific value of $k$, $s$ is the sum map on $V^*=\Sym^n(W^*)$, $m$ is the map from the proof of Lemma \ref{Ind},
$F$ is the map which is the product of polynomials on the base of ${\mathcal E}^{\perp}$'s and the sum in the fibers of 
${\mathcal E}^{\perp}$'s. It is clear that $F$ is well defined and the diagram commutes.

By the proof of Lemma \ref{Ind}, $m_! ({\widetilde L}_{\lambda} \boxtimes{\widetilde L}_{\omega}) =
\bigoplus_{\sigma, |\sigma|=|\lambda|+|\omega|} c_{\lambda, \omega}^{\sigma} {\widetilde L}_{\sigma}$. By proper base change in
the diagram above we see that
$$s_! (\pi_* p^*_{{\mathcal E}^{\perp}} {\widetilde L}_{\lambda} \boxtimes \pi_* p^*_{{\mathcal E}^{\perp}} {\widetilde L}_{\omega})=
\bigoplus_{\sigma, \ |\sigma|=|\lambda|+|\omega|} c_{\lambda, \omega}^{\sigma}  (\pi_* p^*_{{\mathcal E}^{\perp}} {\widetilde L}_{\sigma}).$$
By the proof of Theorem \ref{main}, $\pi_* p^*_{{\mathcal E}^{\perp}} {\widetilde L}_{\mu}$ coincides with ${\bf R}_{\mu}$ on an open
dense subvariety of $\Sec^{|\mu|}$; therefore 
$\pi_* p^*_{{\mathcal E}^{\perp}} {\widetilde L}_{\lambda} * \pi_* p^*_{{\mathcal E}^{\perp}} {\widetilde L}_{\omega}$ coincides with
${\bf R}_{\lambda} * {\bf R}_{\omega}$ on an open dense subvariety of $\Sec^{|\lambda|+|\omega|}$.

By Corollary \ref{tensorproductisperverse}, the tensor product ${\bf L}_{\Lambda} \otimes {\bf L}_{\Omega}$ is a perverse sheaf 
so its Fourier transform ${\bf R}_{\lambda} * {\bf R}_{\omega}$ is perverse as well. Above we detected
all the subquotients in ${\bf R}_{\lambda} * {\bf R}_{\omega}$ whose support is of maximum dimension $|\lambda|+|\omega|$; by the classification of perverse sheaves, 
they are the same as the subquotients of with support of maximum dimension in 
$\bigoplus_{\sigma, |\sigma|=|\lambda|+|\omega|} c_{\lambda, \omega}^{\sigma}  (\pi_* p^*_{{\mathcal E}^{\perp}} {\widetilde L}_{\sigma})$. By the
proof of Theorem \ref{main}, all these subquotients
are of the form ${\bf R}_{\sigma}$ where $|\sigma|=|\lambda|+|\omega|$ and each subquotient ${\bf R}_{\sigma}$ occurs $c_{\lambda, \omega}^{\sigma}$
times. By exactness of the Fourier transform, $c_{\lambda, \omega}^{\sigma}$ must be equal to the multiplicity of the subquotient ${\bf L}_{\Sigma}$ 
in ${\bf L}_{\Lambda} \otimes {\bf L}_{\Omega}$, which is $k_{\Lambda, \Omega}^{\Sigma}$ by Corollary \ref{tensorproductisperverse}.

\end{proof}

\end{document}